\documentclass[12pt]{article}
\usepackage{amssymb}
\usepackage{amsfonts}
\usepackage{amsmath}
\usepackage{array}
\usepackage{enumitem}
\usepackage{tikz}
\usepackage{multicol}
\usepackage{enumitem}
\usepackage{amsthm}
\usepackage{bm}
\usepackage{amssymb}
\usepackage{amsmath}
\usepackage{fullpage} 
\usepackage{url}

\RequirePackage[top=1in,bottom=1in,left=1.4in,right=1.4in,includehead]{geometry}

\usepackage[colorlinks,
linkcolor=blue,
anchorcolor=blue,
citecolor=blue
]{hyperref}

\theoremstyle{plain}
\newtheorem{thm}{Theorem}[section]
\newtheorem{lem}[thm]{Lemma}

\newtheorem{conj}[thm]{Conjecture}

\theoremstyle{definition}

\def\SP{{\mathfrak P}}               
\def\des{\operatorname{des}}         
\def\seg{\operatorname{seg}}         

\newcommand{\sep}{\preceq}

\begin{document}

\begin{center}
	{\large \bf {Multivariate stable Eulerian polynomials \\[5pt]
			on segmented permutations}}
\end{center}

\begin{center}
 Philip B. Zhang$^{1}$ and Xutong Zhang$^{2}$\\[6pt]

	 College of Mathematical Science \\
	Tianjin Normal University, Tianjin  300387, P. R. China\\[6pt]
	
	Email:  $^{1}${\tt  zhang@tjnu.edu.cn} \quad
	$^{2}${\tt  zhang.xutong@foxmail.com}
\end{center}

\noindent\textbf{Abstract.}
Recently, Nunge studied Eulerian polynomials on segmented permutations, namely \emph{generalized Eulerian polynomials}, and further asked  whether their coefficients form unimodal sequences.
In this paper, we prove the stability of the generalized Eulerian polynomials and hence confirm Nunge's conjecture. Our proof is based on Br\"and\'en's stable  multivariate Eulerian polynomials.
By acting  on Br\"and\'en's polynomials with a stability-preserving linear operator, we get a multivariate refinement of the generalized  Eulerian polynomials.
To prove Nunge's conjecture, we  also develop a general  approach to obtain generalized Sturm sequences from bivariate stable polynomials.
\\



\section{Introduction}
The main objective of this paper is to prove a unimodality conjecture for the generalized Eulerian polynomials proposed by Nunge.
In this paper, we give a multivariate refinement of the generalized Eulerian polynomials and then prove  their stability,  from which Nunge's conjecture can be confirmed.


Let $\mathfrak{S}_n$ denote the set of permutations of $[n]:=\{1,2,\ldots,n\}$.
Given a permutation $\pi=\pi_1 \pi_2 \cdots \pi_n \in \mathfrak{S}_n$, let
$$\des (\pi) = | \{i \in [n-1]  : \pi_i > \pi_{i+1}\}|$$
denote the descent number of $\pi$.
The Eulerian number  $A(n,k)$ is defined as  the number of permutations with $k$ descents in $\mathfrak{S}_n$ and the Eulerian polynomial $A_n(t)$ is usually defined as the ordinary generating function of $A(n,k)$, namely,
\begin{align*}
A_n(t) = \sum_{k=0}^{n-1}A(n,k)t^k = \sum_{\pi \in\mathfrak{S}_n}t^{\des (\pi)}.
\end{align*}
Eulerian polynomials and Eulerian numbers have been extensively studied over the
years, see
\cite{Foata2010Eulerian, Petersen2015Eulerian}.

Corteel  and Nunge \cite{Corteel20172} studied a recoil statistic on partially signed permutations during their study of the 2-species  exclusion processes and Hopf algebras over  segmented compositions, see also \cite{Nunge2018equivalence}.
Recently, Nunge \cite{NungeEulerian} gave a corresponding statistic on segmented permutations, namely the descent  statistic.
A \emph{segmented permutation} is a permutation equipped with some separate bars which can be inserted into the  slots between two consecutive letters.
Let  $\SP_n$ denote the set of segmented
permutations of $[n]$.
For $\sigma \in  \SP_n$, a position $i$ is said to be a
\emph{descent} if
$\sigma_{i}>\sigma_{i+1}$ and there is no
bar in the slot between $\sigma_i$ and $\sigma_{i+1}$. We denote by $\des(\sigma)$
(respectively,  $\seg(\sigma)$) the number of descents (respectively, bars) of
$\sigma$. For example, with $\sigma = 2|516|34$, we have
$\des(\sigma) =1$ and $\seg(\sigma) = 2$.
The generalized Eulerian numbers are defined as follows:
\begin{equation*}
T(n, k) = | \{\sigma \in \SP_n : \des(\sigma) = k\}  |
\end{equation*}
and
\begin{equation*}
K(n,i,j) = | \{\sigma \in \SP_n : \des(\sigma)=i,~\seg(\sigma)=j\} |.
\end{equation*}
Following Nunge, let
\begin{equation*}
P_n(t)  =\sum_{k=0}^{n-1}T(n,k)t^k
 = \sum_{\sigma \in \SP_n} t^{\des(\sigma)}
\end{equation*}
and
\begin{equation*}
\alpha_n(t,q)  = \sum_{i=0}^{n-1}\sum_{j=0}^{n-i-1}K(n,i,j) t^i q^j
 = \sum_{\sigma \in \SP_n} t^{\des(\sigma)}q^{\seg(\sigma)}.
\end{equation*}
Note that the polynomial $\alpha_n(t,q)$ gives back the usual Eulerian polynomials at $q=0$ and
the ordered Bell polynomials at $t=0$.
Namely,
$\alpha_n(t,0) = A_n(t)$ and
$\alpha_n(0,q) = \sum_{r=0}^{n-1} (r+1)! S(n,r+1)q^r$,
where $S(n,k)$ are the Stirling numbers of the second kind.

Nunge's conjecture is concerned with the unimodality of the rows and columns of $T(n,k)$ and $K(n,i,j)$.
Recall that a  sequence of positive integers $a_0, a_1, \ldots, a_n$ is said to be \emph{unimodal} if there exists an index $0\le i \le n$ such that $a_0 \le a_1 \le \cdots \le  a_{i-1} \le a_i \ge a_{i+1} \ge \cdots \ge a_n$.
It is well known that, for a sequence of positive numbers, its log-concavity implies
unimodality, see \cite{Stanley1989Log}.
Nunge proposed the following conjecture.
\begin{conj}\label{conj}
	For any positive integer $n$, the integer sequences $\{T(n,k)\}_{k=0}^{n-1}$, \,
	$\{K(n,i,j)\}_{i=0}^{n-j-1}$ and $\{K(n,i,j)\}_{j=0}^{n-i-1}$ are unimodal sequences.
\end{conj}

By the  Newton's inequality (see \cite[p.  104]{Hardy1952Inequalities}), if a polynomial with  nonnegative  coefficients  has  only  real  zeros,  then  the  sequence of its coefficients is log-concave and hence unimodal.
Given a positive integer $n$, for any  $0\le i,j \le n-1$, let
$$K_{n,j}(x)=\sum_{i=0}^{n-1-j}K(n,i,j)x^i \quad \mbox{ and } \quad L_{n,i}(x)=\sum_{j=0}^{n-1-i}K(n,i,j)x^j.$$
In this paper, we obtain the real-rootedness of $P_n(t)$, $K_{n,j}(x)$, and $L_{n,i}(x)$ from which the log-concavity and unimodality of their coefficients can be deduced  and hence give an affirmative answer to  Conjecture \ref{conj}.
\begin{thm}\label{thm:rz}
	For any integers  $n\ge 1$  and  $0\le i,j \le n-1$, the polynomials  $P_n(t)$, $K_{n,j}(x)$ and $L_{n,i}(x)$ have only real zeros.
\end{thm}

Our approach to prove Theorem \ref{thm:rz} is to employ a multivariate stable polynomial, which
generalizes both $P_n(t)$ and $\alpha_n(t,q)$.
The theory of stable polynomials has turned out to play a key role in various combinatorial problems, see  \cite{Braenden2015Unimodality, Wagner2011Multivariate} and references therein.
In this paper, we introduce  a new multivariate polynomial $\bm{\alpha}_n (\bm{x},\bm{y},\bm{z},\bm{w})$ and  show that it can be constructed by acting on Br\"and\'en's multivariate stable Eulerian polynomial with a  linear operator. We then prove this operator preserves stability and hence the polynomial $\bm{\alpha}_n (\bm{x},\bm{y},\bm{z},\bm{w})$ is stable.

Our next step is to deduce the real-rootedness of  $P_n(t)$, $K_{n,j}(x)$, and $L_{n,i}(x)$  from the stability of   $\bm{\alpha}_n (\bm{x},\bm{y},\bm{z},\bm{w})$.
We note that $\alpha_n(t,q)$ can be reduced from  $\bm{\alpha}_n (\bm{x},\bm{y},\bm{z},\bm{w})$ and $P_n(t) =  \alpha_n(t,1)$.   Since the involved operators preserve real stability, we obtain the stability of  $\alpha_n(t,q)$ and $P_n(t)$.
Hence, the univariate stable polynomial $P_n(t)$ has only real zeros.
To prove the real-rootedness of $K_{n,j}(x)$ and $L_{n,i}(x)$, we develop a general way (Theorem \ref{thm:z}) to obtain generalized Sturm sequences (defined in Section \ref{sec:prelim}) from bivariate stable polynomials.


The remainder of this paper is organized as follows.
In Section~\ref{sec:prelim}, we recall some definitions and results on stable polynomials, including certain linear operators which preserve real stability. We also give a new result (Theorem \ref{thm:z}) which relates  bivariate stable polynomials to generalized Sturm sequences.
Section~\ref{sec:proof} is dedicated to our proof of Theorem \ref{thm:rz}. Our proof is based on the stability of $\bm{\alpha}_n (\bm{x},\bm{y},\bm{z},\bm{w})$.
We also relate $\alpha_n(t,q)$ and $P_n(t)$ to the classical Eulerian polynomial $A_n(t)$ and then prove the stability of $\alpha_n(t,q)$ in an alternative way.
\section{Preliminaries}
\label{sec:prelim}

In this section, we shall give an overview of stable polynomials.
After recalling the definition of stable polynomials, we list some stability-preserving linear operators  which will be used in the paper. We also present Theorem \ref{thm:z}, a general approach to obtain generalized Sturm sequences from bivariate stable polynomials. Our proof of Theorem \ref{thm:z} is similar to that of Newton's inequality and  based on the  Hermite--Biehler Theorem.


Now let us  recall the notion of real stability, which generalizes the
notion of real-rootedness from  univariate real polynomials to multivariate real polynomials.
For a positive integer $n$, let $\bm{x}$ be the $n$-tuple $(x_1, \dots, x_n)$.
Let $\mathbb{H}_+ = \{z \in \mathbb{C} : \mbox{Im}(z)>0 \}$ denote the open upper complex
half-plane.
A polynomial $f \in \mathbb{R}[\bm{x}]$ is said to be (real) \emph{stable} if $f(\bm{x}) \neq 0$  for any
$\bm{x} \in \mathbb{H}_+^n$ or $f$ is identically zero.
Note that a univariate polynomial $f(x) \in \mathbb{R}[x]$ is stable if and only if it has real zeros.
A polynomial $f(\bm{x})$ is said to be \emph{multiaffine} if the power of each indeterminate $x_i$ is at most one.
For a set $\mathcal{P}$ of polynomials, let $\mathcal{P}^{MA}$ be the set of multiaffine polynomials in $\mathcal{P}$.
Borcea and Br{\"a}nd{\'e}n~\cite{Borcea2009Lee} gave a complete characterization of the linear operators which preserve multivariate stable polynomials.
In this paper, we shall use a multiaffine version of  Borcea and Br{\"a}nd{\'e}n's characterization, which shall play a key role in this  paper to prove the stability of  polynomials.

\begin{lem}[{\cite[Theorem 3.5]{Wagner2011Multivariate}}]
	\label{lem:bb}
	Let $T: \mathbb{R}[\bm{x}]^{MA} \to \mathbb{R}[\bm{x}]$ be a linear operator acting on the
	variables $\bm{x}=(x_1, \dots, x_n)$.  If the polynomial
	$$T\left(\prod_{i=1}^{n}(x_i+\hat{x}_i) \right)$$
	is a   stable polynomial  of variables $\bm{x}$ and $\bm{\hat{x}}=(\hat{x}_1,\hat{x}_2,\ldots,\hat{x}_n)$,
	then $T$ preserves real stability.
\end{lem}

Once multivariate  polynomials are shown be stable, we can then reduce
them to real stable univariate polynomials by using the following operations.

\begin{lem}[{\cite[Lemma 2.4]{Wagner2011Multivariate}}]\label{lem:stable}
	Given $i,j \in [n]$, the following operations preserve real stability of $f \in \mathbb{R}[\bm{x}]$:
	\begin{enumerate}
		\item \emph{Differentiation:} $f \mapsto \partial f/\partial x_i.$
		\item \emph{Diagonalization:} $f \mapsto f|_{x_i=x_j}.$
		\item	\emph{Specialization:} for $a \in \mathbb{R}$, $f\mapsto f|_{x_i = a}.$
	\end{enumerate}
\end{lem}

Given two real-rooted polynomials $f(x)$ and $g(x)$ with positive leading coefficients,  we say that  {$g(x)$ \textit{interlaces} $f(x)$}, denoted $g(x)\sep f(x)$, if
\begin{align*}
\cdots\le s_2\le r_2\le s_1\le r_1,
\end{align*}
where $\{r_j\}$ and $\{s_k\}$ are the sets of zeros of $f(x)$ and $g(x)$, respectively.
A sequence $\{f_n(x)\}_{n\geq 0}$ of real polynomials with positive leading coefficients is said to  be \emph{a generalized Sturm sequence} if  $\deg f_n(x)=n$  and $f_n(x)\sep f_{n+1}(x)$.

The following theorem  provides a general approach to obtain generalized Sturm sequences from bivariate stable polynomials.

 \begin{thm}\label{thm:z}
Suppose that  $F(x,y)=\sum_{j=0}^n f_j(x)y^j$ is a bivariate polynomial with real coefficients. If  $F(x,y)$ is stable, then  the polynomial $f_j(x)$ has only real zeros  for any $0\le j \le n$ and moreover $\{f_{n-j}(x)\}_{j=0}^{n}$ forms a generalized Sturm sequence.
\end{thm}


In order to prove Theorem \ref{thm:z}, we need the Hermite--Biehler Theorem, which reveals the close connection between interlacing and stability.
\begin{lem}[{Hermite--Biehler, \cite[Th. 6.3.4]{Rahman2002Analytic}}]\label{lem:HB}
Let $f(z)$ and $g(z)$ be two non-constant polynomials with real coefficients.
 Then the following statements are equivalent:
 \begin{itemize}
 	\item  $f(z)$ and $g(z)$ have only real zeros, and moreover $g(z)  \sep f(z)$;
 	\item the polynomial $f(z)+ig(z)$ is stable.
 \end{itemize}
\end{lem}

Now we are at the position to give  a proof of Theorem \ref{thm:z}.
Note that our proof is similar to that of Newton's inequality, see \cite[p.  104]{Hardy1952Inequalities}.

\begin{proof}[Proof of Theorem \ref{thm:z}]
By taking the $k$-th order partial derivative with respect to $y$ of the real stable polynomial $$F(x,y)=\sum_{j=0}^n f_j(x)y^j,$$ it follows from Lemma \ref{lem:stable} that
$$\sum_{j=k}^{n} (j)_{k}f_j(x)y^{j-k} $$
is real stable, where $(j)_k=j(j-1)\cdots (j-k+1)$.
Note that if $y\in \mathbb{H}_+$ then $-\cfrac{1}{y} \in \mathbb{H}_+$. Hence, we obtain the real stability of
\begin{align}\label{eq:r}
	y^{n-k} \sum_{j=k}^{n} (j)_{k}f_j(x)(-\frac{1}{y})^{j-k} = \sum_{j=k}^{n} (-1)^{j-k} (j)_{k}f_j(x) y^{n-j}.
\end{align}
Similarly, by taking the $(n-k-1)$-th  order the partial derivative  with respect to $y$ of \eqref{eq:r}, we get the real stability of
 \begin{align*}
 	& \sum_{j=k}^{k+1} (-1)^{j-k} (j)_{k} (n-j)_{n-k-1} f_j(x) y^{k-j+1} \\
 	= &\  k!(n-k)! f_k(x)y - (k+1)! (n-k-1)! f_{k+1}(x) \\
 	= &\   (k+1)!(n-k-1)! \left(  \frac{n-k}{k+1} f_k(x)y  -  f_{k+1}(x) \right).
 \end{align*}
Replacing $y$ by $\cfrac{k+1}{n-k}\, y$, it follows that
$$f_{k+1}(x)-f_k(x)y$$
is real stable and so is
$$y(f_{k+1}(x)-f_k(x)(-\frac{1}{y}))= f_k(x)+ y f_{k+1}(x).$$
Therefore, it follows from Lemma \ref{lem:HB} that  $f_{k+1}(x) \sep f_k(x)$.
This completes the proof of Theorem \ref{thm:z}.
\end{proof}

\section{Proof of  Theorem \ref{thm:rz}}\label{sec:proof}
The main objective of this section is to prove Theorem \ref{thm:rz}.
We first recall Br\"and\'en's  multivariate Eulerian polynomial and then introduce a multivariate refinement of Nunge's Eulerian polynomials on segmented permutations.
We next show that it can be obtained by acting on Br\"and\'en's  multivariate Eulerian polynomial with  a stability-preserving linear operator.

Before presenting our result, let us first recall Br\"and\'en's  multivariate Eulerian polynomial.
Given a permutation $\pi \in \mathfrak{S}_n$, let
\begin{align*}
\mathcal{DT} (\pi) & = \{\pi_i : \pi_i > \pi_{i+1}\} \quad  \mbox{and}\\[5pt]
\mathcal{AT} (\pi) & = \{\pi_{i+1} :  \pi_i<\pi_{i+1}\}
\end{align*}
be  the  \emph{descent top} set and the \emph{ascent top} set, respectively.
For $\mathcal{T}$ a set with entries from $[n]$, we let $\bm{x}^{\mathcal{T}} = \prod_{i \in \mathcal{T}} x_i$.
Br\"and\'en defined a  real multivariate  polynomial  $\bm{A}_n (\bm{x},\bm{y})$ as follows:
\begin{equation}\label{eq:branden}
\bm{A}_n (\bm{x},\bm{y}) = \sum_{\pi \in \mathfrak{S}_n}  w(\pi),
\mbox{  where  }  w(\pi) =
\bm{x}^{\mathcal{DT} (\pi)} \bm{y}^{\mathcal{AT} (\pi)}.
\end{equation}
For example, $w(251634) = y_5 x_5 y_6 x_6 y_4$.
Clearly, $\bm{A}_n (\bm{x},\bm{y})$ is multiaffine.
Br\"and\'en  proved the following result.
\begin{lem}\label{thm:branden}
	For any positive integer $n$, the polynomial $\bm{A}_n (\bm{x},\bm{y})$
	is stable.
\end{lem}
For a proof of Lemma \ref{thm:branden} and further generalizations, we refer the reader  to~\cite{Braenden2011Proof, ChenContext,  Haglund2012Stable, Visontai2013Stable}.
By diagonalizing the  variables $x_i$ to $x$,  specializing $y_i$ to $1$, it follows that
	\[A_n(x) = \sum_{\pi \in  \mathfrak{S}_n} x^{|\mathcal{DT} (\pi)|} = \sum_{\pi \in  \mathfrak{S}_n} x^{\des (\pi)}\] is stable.
Since $A_n(x)$ is univariate, it is equivalent to say that $A_n(x)$ has only real zeros.

We next give our multivariate Eulerian polynomials on segmented permutations.
Given a segmented permutation $\sigma \in \SP_n$, let
	\begin{align*}
	\mathcal{DT} (\sigma) & = \{\sigma_i :  \sigma_i > \sigma_{i+1} \mbox{ and there is no bar in the slot  between } \sigma_i \mbox{ and } \sigma_{i+1} \},\\[5pt]
	\mathcal{AT} (\sigma) & = \{\sigma_{i+1} : \sigma_i < \sigma_{i+1} \mbox{ and there is no bar in the slot between } \sigma_i \mbox{ and } \sigma_{i+1} \},\\[5pt]
	\mathcal{DTS} (\sigma) & = \{\sigma_i : \sigma_i > \sigma_{i+1} \mbox{ and there is a bar in the slot between } \sigma_i \mbox{ and } \sigma_{i+1} \}, \mbox{ and}\\[5pt]
	\mathcal{ATS} (\sigma) & = \{\sigma_{i+1} : \sigma_i < \sigma_{i+1} \mbox{ and there is a bar in the slot  between } \sigma_i \mbox{ and } \sigma_{i+1} \}
\end{align*}
be  the  \emph{descent top} set,  the \emph{ascent top} set,  the \emph{descent top segment} set and the \emph{ascent top segment} set, respectively.
Let $\bm{\alpha}_n (\bm{x},\bm{y},\bm{z},\bm{w})$ be a real multivariate multiaffine polynomial defined as
\begin{equation}\label{eq:main}
\bm{\alpha}_n (\bm{x},\bm{y},\bm{z},\bm{w})
= \sum_{\sigma \in \SP_n}  w'(\sigma),  \quad
\mbox{  where  }  w'(\sigma) =
\bm{x}^{\mathcal{DT} (\sigma)} \bm{y}^{\mathcal{AT} (\sigma)}
\bm{z}^{\mathcal{DTS} (\sigma)} \bm{w}^{\mathcal{ATS} (\sigma)}.
\end{equation}
For example, $w' (2|516|34) = w_5 x_5 y_6 z_6 y_4$.
The main result of this section is stated as follows.
\begin{thm}\label{thm:stable}
	For any positive integer $n$,  the polynomial $\bm{\alpha}_n (\bm{x},\bm{y},\bm{z},\bm{w})$ is stable.
\end{thm}

\begin{proof}
To prove  this theorem, we first establish an identity, which relates the polynomial $\bm{\alpha}_n (\bm{x},\bm{y},\bm{z},\bm{w})$ to Br\"and\'en's polynomial $\bm{A}_n (\bm{x},\bm{y})$.
From the definition of $\bm{A}_n (\bm{x},\bm{y})$, it is clear to see that for any permutation $\pi$ there is a one-to-one correspondence between variables appearing in $w(\pi)$ and slots of two adjacent letters in  $\pi$.
Given a permutation $\pi \in \mathfrak{S}_n$, we can generate a segmented permutation $\sigma \in  \SP_n$ by deciding whether to insert a bar in every slot between two adjacent letters.
If a bar is inserted in a slot followed by a descent top (respectively, preceding an ascent top), namely $i$, then the descent top (respectively, the ascent top) will be replaced by a descent top segment (respectively,  an ascent top segment) and hence the corresponding variable $x_i$ (respectively, $y_i$) will be replaced by $z_i$ (respectively, $w_i$).
Then, for any segmented permutation $\sigma$ there is  a one-to-one correspondence between  variables appearing in $w'(\sigma)$ and slots of two adjacent letters  in $\sigma$.
Hence, 
together with the fact that both $\bm{A}_n (\bm{x},\bm{y})$ and $\bm{\alpha}_n (\bm{x},\bm{y},\bm{z},\bm{w}) $ are multiaffine, 
we obtain that
\begin{align}\label{eq:aA}
	\bm{\alpha}_n (\bm{x},\bm{y},\bm{z},\bm{w}) = \prod_{j=2}^{n}   (1+ w_j \frac{\partial}{\partial y_j}) (1+z_j \frac{\partial}{\partial x_j}) \bm{A}_n (\bm{x},\bm{y}).
\end{align}

We next prove the stability of $\bm{\alpha}_n (\bm{x},\bm{y},\bm{z},\bm{w})$ via \eqref{eq:aA}.
For  $2\le j \le n$, let $T_j= 1+z_j \frac{\partial}{\partial x_j}$  be a linear operator defined on $\mathbb{R}[\bm{x},\bm{y},\bm{z},\bm{w}]^{MA}$.
Since the polynomial
\begin{align*}
&T_j \big(  \prod_{i=1}^{n} (x_i+ \hat{x}_i) (y_i+ \hat{y}_i)
(z_i+ \hat{z}_i) (w_i+ \hat{w}_i)  \big)\\
=  &\   (x_j+\hat{x}_j+ z_j)  \prod_{k\neq j} (x_k+\hat{x}_k) \times \prod_{i=1}^{n} (y_i+ \hat{y}_i)
(z_i+ \hat{z}_i) (w_i+ \hat{w}_i)
\end{align*}
is stable,
it follows from Lemma~\ref{lem:bb} that the linear operator  $T_j$ preserves  stability, 
Similarly, the linear operator  $1+w_j \frac{\partial}{\partial y_j}$ defined on $\mathbb{R}[\bm{x},\bm{y},\bm{z},\bm{w}]^{MA}$ also preserves stability for any $2\le j \le n$. Therefore, we obtain that their product  $\prod_{j=2}^{n}   (1+ w_j \frac{\partial}{\partial y_j}) (1+z_j \frac{\partial}{\partial x_j})$ preserves stability.
Since $\bm{A}_n (\bm{x},\bm{y})$ is stable, we get the desired stable property of $\bm{\alpha}_n (\bm{x},\bm{y},\bm{z},\bm{w})$ by~\eqref{eq:aA}. This completes the proof.
\end{proof}

By diagonalizing the  variables $x_i$ to $x$,  specializing $y_i$ to $1$, and substituting $z_i$ and $w_i$ by $t$, it follows that

\begin{thm}\label{thm:as}
	For any positive integer $n$, the polynomial
\begin{align*}
	\alpha_n(t,q) = \sum_{\sigma \in \SP_n} t^{\des(\sigma)}q^{\seg(\sigma)}
\end{align*}
	is stable.
\end{thm}

Now it is time for us to prove Theorem \ref{thm:rz}.

\begin{proof}[Proof of Theorem \ref{thm:rz}]
	Since $\alpha_n(t,1) = P_n(t)$, if follows that  $P_n(t)$ is stable and hence real-rooted as a univariate polynomial with real coefficients.
	For  the real-rootedness of $K_{n,j}(x)$ and $L_{n,i}(x)$, we shall apply  Theorem \ref{thm:z} to $\alpha_n(t,q)$.
	Since
	$$\alpha_n(t,q) =  \sum_{j=0}^{n-1}K_{n,j}(t)q^j,$$
	we get that the polynomial sequence $\{K_{n,n-j-1}(t)\}_{j=0}^{n-1}$ forms a generalized Sturm sequence.
	Similarly, since
	$$\alpha_n(t,q) =   \sum_{i=0}^{n-1}L_{n,i}(q)t^i,$$
	we get that the polynomial sequence $\{L_{n,n-i-1}(t)\}_{i=0}^{n-1}$ forms a generalized Sturm sequence.
	This completes the proof of Theorem \ref{thm:rz}.
\end{proof}


Before ending this paper, we would like to express $\alpha_n(t,q)$ and $P_n(t)$ in terms of $A_n(t)$, which leads to an alternative proof of Theorem \ref{thm:as}.

\begin{thm}
	For any positive integer $n$, we have
	\begin{align}\label{eq:alphaA}
		\alpha_n(t,q) = (1+q)^{n-1} A_n \left(\frac{t+q}{1+q}\right),
	\end{align}
	and
	\begin{align}\label{eq:PA}
		P_n(t) = 2^{n-1} A_n \left(\frac{t+1}{2}\right).
	\end{align}
\end{thm}
\begin{proof}
By the following identity \cite[Corollary 2.5]{NungeEulerian}:
\begin{align*}
K(n,i,j) = \sum_{k=0}^{n-1} \binom{k}{i} \binom{n-1-k}{i+j-k} A(n,k) ,
\end{align*}
we have that
\allowdisplaybreaks
\begin{align*}
\alpha_n(t,q)  & =  \sum_{i=0}^{n-1}\sum_{j=0}^{n-i-1}\sum_{k=0}^{n-1} \binom{k}{i} \binom{n-1-k}{i+j-k} A(n,k)  t^i q^j \\
& = \sum_{k=0}^{n-1} A(n,k)  \sum_{i=0}^{k} \binom{k}{i} t^i  q^{k-i} \sum_{j=k-i}^{n-i-1} \binom{n-1-k}{i+j-k} q^{i+j-k}\\
& = \sum_{k=0}^{n-1} A(n,k)  \sum_{i=0}^{k} \binom{k}{i} t^i  q^{k-i} (1+q)^{n-1-k}\\
& = \sum_{k=0}^{n-1} A(n,k)  (1+q)^{n-1-k} q^{k} \sum_{i=0}^{k} \binom{k}{i} (t/q)^i  \\
& = \sum_{k=0}^{n-1} A(n,k)  (1+q)^{n-1-k} q^{k} (1+t/q)^k  \\
& = \sum_{k=0}^{n-1} A(n,k)    (1+q)^{n-1-k} (t+q)^k\\
& = (1+q)^{n-1}\sum_{k=0}^{n-1} A(n,k)   \left(\frac{t+q}{1+q}\right)^k.
\end{align*}
This completes the proof of \eqref{eq:alphaA}.
Since $P_n(t)=\alpha_n(t,1)$, the equation \eqref{eq:PA} follows from \eqref{eq:alphaA}.
This completes the proof.
\end{proof}

\begin{proof}[Alternative proof of Theorem \ref{thm:as}]
Let $t \in \mathbb{H}_+$ and $q\in \mathbb{H}_+$. Then we have that $t+q \in \mathbb{H}_+$ and  $1+q \in \mathbb{H}_+$ and thus $\frac{t+q}{1+q}$ will not be a negative real number. Since the Eulerian polynomial $A_n(x)$ has only negative real zeros, the polynomial $\alpha_n(t,q)$ will  be non-zero whenever $t \in \mathbb{H}_+$ and $q\in \mathbb{H}_+$.
Therefore, the polynomial  $\alpha_n(t,q)$  is stable.
\end{proof}

We remark  that comparing \eqref{eq:PA} with the well-known formula of Eulerian polynomials:
\begin{align*}
\frac{A_n(t)}{(1-t)^{n+1}} = \sum_{k=0}^{\infty} (k+1)^n t^{k},
\end{align*}
it follows that
\begin{align*}
\frac{P_n(t)}{(1-t)^{n+1}} = \sum_{k=1}^{\infty} \left(1+t\right)^{k-1} \frac{k^n}{2^{k+1}}
\end{align*}
which has appeared in \cite[Proposition 3.5]{NungeEulerian}.

\vskip 3mm
\noindent {\bf Acknowledgments.}
We would like to thank the anonymous referee for the valuable comments. 
This work was supported by the National Science Foundation of China (Nos. 11626172, 11701424), the PHD Program of TJNU (No. XB1616), and  MECF of Tianjin (No. JW1713).


\end{document}